\documentclass[a4paper,reqno,11pt]{amsart}
\usepackage{t1enc}
\usepackage[margin=1in]{geometry}
\usepackage{ifthen}
\usepackage{graphicx}

\newcommand{\R}{\mathbf{R}}

\newcommand{\pr}{\textbf{P}}
\newcommand{\ex}{\mathbf{E}}

\usepackage{color}

\theoremstyle{plain}
\newtheorem{theorem}{Theorem}

\newtheorem{corollary}{Corollary}
\newtheorem{proposition}{Proposition}

\theoremstyle{definition}

\theoremstyle{remark}

\newcommand{\formula}[2][nolabel]
{\ifthenelse{\equal{#1}{nolabel}}
 {\begin{align*} #2 \end{align*}}
 {\ifthenelse{\equal{#1}{}}
  {\begin{align} #2 \end{align}}
  {\begin{align} \label{#1} #2 \end{align}}
 }
}

\numberwithin{equation}{section}

\begin{document}

%
%                            ---------- o ----------
%

\title []{Wallach sets and squared Bessel particle systems}
\thanks{Jacek Ma\l{}ecki was supported by the National Science Centre (Poland) grant no. 2013/11/D/ST1/02622.}
\author{Piotr Graczyk, Jacek Ma{\l}ecki}
\address{Piotr Graczyk \\ LAREMA \\ Universit\'e d'Angers \\ 2 Bd Lavoisier \\ 49045 Angers cedex 1, France}
\email{piotr.graczyk@univ-angers.fr}
\address{  Jacek Ma{\l}ecki,  \\ Faculty of Pure and Applied Mathematics\\Wroc{\l}aw University of Technology \\ ul. Wybrze{\.z}e Wyspia{\'n}\-skiego 27 \\ 50-370 Wroc{\l}aw, Poland}
\email{jacek.malecki@pwr.wroc.pl }

\begin{abstract}
We determine the classical and the non-central  Wallach sets $W_0$ and $W$ by classical probabilistic methods. We prove the Mayerhofer conjecture on $W$. We exploit the fact that $(x_0,\beta)\in W$ if and only if $x_0$ is the starting point and $2\beta$ is the drift of a squared Bessel matrix process $X_t$ on the cone $\overline{Sym^+(\R,p)}$. Our methods are based on the study of SDEs  for the symmetric polynomials of $X_t$ and for the eigenvalues of  $X_t$, i.e. the squared Bessel particle systems.
\end{abstract}
\maketitle
%
%                            ---------- o ----------
%
\section{Introduction and Preliminaries}
\label{}
The aim of this paper is to prove the characterization of the non-central Wallach set $W$, conjectured in \cite{bib:mayerJMA} by Mayerhofer. More precisely, let us denote by $\mathcal{S}_p=Sym(\R,p)$ the space of symmetric $p\times p$ matrices and let $\mathcal{S}_p^+$ be the open cone of positive definite matrices. The (central) Wallach set $W_{0}$ is defined as the set of admissible $\beta\in\R$ such that there exists a random matrix $X$ with values in $\bar {\mathcal S}^+_p$ (equivalently a measure with support in $\bar {\mathcal S}^+_p$) such that its Laplace transform is of the form
 $$
 \ex e^{-{\bf Tr}(uX)}=(\det(I+2\Sigma u))^{-\beta},\ \  u\in  {\mathcal S}^+_p,
 $$
where $\Sigma\in {\mathcal   S}_p^+$. It is well-known (see  \cite{bib:farautKOR}, pp. 137, 349) that
$$
W_0=\frac12 B \cup \left[\frac{p-1}2,\infty\right)   \/,
$$
where $B=\{0,1,\cdots,p-2\}$. However, a similar question can be stated in a more general setting. Let $x_0\in \bar  {\mathcal  S}_p^+$ and $\beta \in \R$.  We say  that the pair $(x_0,\beta)$ belongs to the  non-central Wallach set $W$
if there exists a random matrix $X$ with values in $\bar {\mathcal  S}_p^+$ having the Laplace transform 
  \begin{equation}
     \label{eq:Lap1}
     \ex e^{-{\bf Tr}(uX) }= (\det(I+2\Sigma u))^{-\beta} \exp[- {\bf Tr}(x_0(I+2\Sigma  u)^{-1}u))],  \ \  u\in  {\mathcal S}^+_p,
  \end{equation}
	for a matrix $\Sigma\in {\mathcal   S}_p^+$. The interest in random matrices verifying \eqref{eq:Lap1} comes from the fact that if
	$$
	  X = \xi_1\xi_1^T+\ldots+\xi_n\xi_n^T=q(\xi)\/,\quad  \xi=(\xi_1,\ldots,\xi_p)\/,
	$$
	where $\xi_{i}\sim N_p(m_i,\Sigma)$ are independent normal vectors in $\R^p$, then the Laplace transform of $X$ is given by \eqref{eq:Lap1} with $\beta =n/2$ and $x_0=q(m_1,\ldots,m_n)$ (see \cite{bib:mayer}). Consequently, random matrices $X$ verifying \eqref{eq:Lap1} are of great importance in statistics as estimators of the normal covariance matrix $\Sigma$. Obviously $(0,\beta)\in W$ if and only if $\beta \in W_0$. Note also that whenever $(x_0,\beta)\in W$ then we have $\beta\geq 0$, otherwise  $\ex e^{-{\bf Tr}(uX) }$ would be unbounded (take for example $u=nI$, $n\in\mathbb{N}$, $n\to \infty$).  
		
	The characterization of the non-central set Wallach set $W$ has been recently studied by Letac and Massam in \cite{bib:LetMassFalse}. However, \cite{bib:LetMassFalse} contains an error in the formulation of the result and in its proof, which was pointed out by Mayerhofer\cite{bib:mayerJMA}.  Mayerhofer stated in \cite{bib:mayerJMA} the following conjecture 
		
\medskip

   {\bf Mayerhofer Conjecture}. {\it The non-central Wallach set is characterized by }
   $$
  (x_0,\beta )\in W \ \ \Leftrightarrow \ \ (\beta\in  \left[\frac{p-1}2,\infty\right), x_0\in \bar  {\mathcal   S}_p^+)  \ {\it or}\ (2\beta\in B, rk(x_0)\le 2\beta).
	$$
	
	Sufficiency of these conditions  was showed by Bru in \cite{bib:b91}, except the case $2\beta=p-1$, that may be found in 
	\cite{bib:LetMassFalse} and \cite{bib:gm11}. In what concerns the necessity, Mayerhofer  proved in  \cite{bib:mayerJMA} that if $(x_0,\beta)\in W$ and $2\beta\in B$, then $rk(x_0)\le 2\beta +1$. In this note  we provide a simple proof of the Mayerhofer conjecture based on It\^o stochastic calculus. Completely different approach based on  analytical methods was proposed in the unpublished note {\cite{bib:LetMass}}.
	
 Denoting more precisely the set of $(x_0,\beta)$ with the property \eqref{eq:Lap1} by $W_\Sigma$, it is easy 
 to show (\cite{bib:mayer}, Prop.III.5.1) that $(x_0,\beta)\in W_\Sigma$ if and only if $(\Sigma^{-\frac12}x_0\Sigma^{-\frac12},\beta)\in W_I$. 
 Thus the  conditions of  Mayerhofer's conjecture are the same  for any $\Sigma\in {\mathcal   S}_p^+$ and, in the sequel, we will only consider the case $\Sigma=I$.

Our  main tools are the results of the article \cite{bib:gm2}, that we adapt to  the set-up of   BESQ matrix SDEs 
\eqref{eq:Wishart:SDe}
and their eigenvalue processes, i.e. BESQ particle systems.
These new results on  BESQ particle systems are interesting independently and are an income to the study of these processes started in 
\cite{bib:katori2011}.

\section{Wallach sets and Stochastic Analysis}
We begin this section by recalling (following the exposure given in \cite{bib:mayer}, \cite{bib:mayerJMA}) the relations between Wallach sets and matrix squared Bessel processes. Consequently, we translate the Mayerhofer conjecture to the question of the existence of solutions in $\bar{\mathcal{S}}_p^{+}$ of the appropriate matrix SDE (depending on $\beta$) starting from $x_0$. Then, using symmetric polynomials method and comparison theorem we show that such solutions exist only if $(x_0,\beta)$ fulfills the conditions stated by Mayerhofer.

 \subsection{  Wallach  sets and stochastic processes}
  
 Let $W_t$ be a Brownian matrix of dimension $p\times p$. We call {\it matrix BESQ process} any solution of the following SDE
 \begin{eqnarray}
\label{eq:Wishart:SDe}
   dX_t = \sqrt{|X_t|}dW_t+dW^T_t\sqrt{|X_t|}+\alpha Idt\/,\quad    X_t\in  {\mathcal S}_p\/,t\ge 0;\quad  X_0=x_0.
	\end{eqnarray}
  
Recall that if $g:\R\mapsto\R$ then $g(X)$ is defined spectrally, i.e. $g(U diag(\lambda_i) U^T)=U diag(g(\lambda_i)) U^T$, where $U\in SO(p)$.
 When $X_0=x_0\in \bar {\mathcal  S}_p^+$, such processes were studied by Bru in \cite{bib:b91}, who among others showed that if (\ref{eq:Wishart:SDe}) admits a solution $X_t \in \bar {\mathcal  S}_p^+$ with  $X_0=x_0\in   \bar{\mathcal   S}^+_p$, then the Laplace transform of $X_t$ is 
\begin{equation}
 \label{Lap_Wish} 
 \ex^{x_0}[\exp(- {\bf Tr}(uX_t)]=(\det(I+2tu))^{-\alpha/2} \exp[- {\bf Tr}(x_0(I+2tu)^{-1}u))],\quad u\in  {\mathcal S}^+_p\/.
   \end{equation}
	In particular, by taking $t=1$, it means that $(x_0, \frac{\alpha}{2})\in W$. Mayerhofer showed in \cite{bib:mayer} and \cite{bib:mayerJMA} that in fact these properties are equivalent.	
	
\begin{proposition}[Mayerhofer]
The stochastic differential equation (\ref{eq:Wishart:SDe}) with $x_0\in \bar {\mathcal  S}_p^+$ has a solution in  $\bar {\mathcal  S}_p^+$ if and only if $(x_0,\frac{\alpha}2)\in W$. 
\end{proposition}

%%%%%%%%%%%%%%%%%%%%%%%%%%%%%%%%%%%%%%%%%%%%%%%%%%%%%%%%%%%%%%%%%%%%%%%%%%%%%%%%
%%%%%%%%%%%%%%%%%%%%%%%%%%%%%%%%%%%%%%%%%%%%%%%%%%%%%%%%%%%%%%%%%%%%%%%%%%%%%%%%%

\subsection{Symmetric polynomials of solutions of matrix SDEs}\label{sec:polyn}

If $X$ is a symmetric $p\times p$ matrix, we define  the  polynomials $e_n(X)$
as basic symmetric polynomials 
 $$
   e_n(X) = \sum_{i_1<\ldots<i_n}\lambda_{i_1}(X)\lambda_{i_2}(X)\ldots \lambda_{i_n}(X)\/,\ \ \ \quad n=1,\ldots,p;
 $$
in the eigenvalues $\lambda_1(X) \le  \ldots\le\lambda_p(X)$ of $X$. Moreover, we use the convention that $e_0(X)\equiv 1$. Up to the sign change, the polynomials $e_n$ are the coefficients of the characteristic polynomial of $X$, i.e.
$$
\det(X-uI)=(-1)^p  u^p + (-1)^{p-1} e_1(X)u^{p-1}+\ldots  -e_{p-1}(X)u+e_p(X) 
$$
and are polynomial functions of the entries of the matrix $X$. In particular, $e_p(X)=\det X$. In \cite{bib:gm2}, the symmetric polynomials related to general class of non-colliding particle systems were studied in details. Using the results therein we get the following characterization of the symmetric polynomials related to matrix squared Bessel processes.

\begin{proposition}
\label{prop:Poly}
Let $X_t$ be a solution of the matrix SDE \eqref{eq:Wishart:SDe} and  $X_t\in  {\mathcal S}^+_p,t\ge 0$. Then the
 symmetric polynomials $e_n(t):=e_n(X_t)$, $n=1,\ldots, p$ are semimartingales satisfying the following system of SDEs
\begin{eqnarray}
   de_n &=& M_n(e_1,\ldots,e_p)dV_n +(p-n+1)(\alpha-n+1)e_{n-1}dt\/,\quad n=1,\ldots,p-1\/, \label{eq:polynom_first:SDEs} \\
   de_p &=& 2\sqrt{e_{p-1}e_p}dV_p +(\alpha-p+1)e_{p-1}dt,  \label{eq:polynom_last:SDEs}  
	\end{eqnarray}
where  $V_i$, $i=1,\ldots, p$  are  one-dimensional Brownian motions and the functions $M_n$ are continuous on $\R^p$.
	\end{proposition}
	
	\begin{proof}
	
	Note that the explicit forms of the martingale parts $M_n(e_1,\ldots,e_p)dV_n$  as well as their brackets $d\left<e_n,e_{m}\right>$  are known  for every $n,m=1,\ldots,p$ (see Proposition 3.2 in \cite{bib:gm2}). However, they will not be used in the sequel apart from the case $n=p$ stated explicitly in Proposition \ref{prop:Poly} above.
	
	The symmetric polynomials $(e_1,\ldots, e_n)$ are given by  an analytic function (polynomials of the coefficients) of the matrix $X$. Thus   It\^o formula, applied to the SDE for the matrix process $X_t$, gives a system of the SDEs  for $(e_1,\ldots, e_n)$. We determine these SDEs like in Propositions 3.1 and 3.2 in \cite{bib:gm2}, using Theorem 3 from \cite{bib:gm11} in the case when eigenvalues of $x_0$ are all distinct.
Evidently, by It\^o formula, this form of the SDEs system describing $(e_1,\ldots,e_p)$ does not depend on the starting point $x_0$, i.e. it does not change if we remove the condition that eigenvalues of the initial points are all different. We write $e_{n}^{\overline i}$ for the incomplete polynomial of order $n$, not containing the variable $\lambda_i(e)$; the notation $e_{n}^{\overline i,\overline j}$ is analogous. Using  formulas from Proposition 3.2 in \cite{bib:gm2} we find that
	$$
	M_n=2\left(\sum_{i=1}^p|\lambda_i|(e_{n-1}^{\overline i})^2\right)^{1/2} 
	$$
	and, in particular,
	 when  $X_t\in  {\mathcal S}^+_p$, $t\ge 0$, (i.e. the eigenvalues are non-negative),
	we obtain $M_p= 2\sqrt{e_{p-1}e_p}$. Moreover, we have the following expressions for the drift parts of $de_n$:
	$$ 
	\sum_{i=1}^p \alpha e_{n-1}^{\overline i}-\sum_{i<j}(|\lambda_i|+|\lambda_j|)e_{n-2}^{\overline{i},\overline j}=(p-n+1)(\alpha-n+1)e_{n-1},
	$$
	where we removed the absolute values since we assumed that all the eigenvalues are non-negative. This ends the proof.	

\end{proof}

The fact that matrix BESQ $X_t$ leaves $\bar{\mathcal{S}}_p^+$ is controlled by $e_p$, which is the determinant of $X_t$, i.e. if $e_p(t)$ is negative then $X_t$ cannot be in $\bar{\mathcal{S}}_p^+$. The explicit formulas for the SDEs describing the symmetric polynomials $e_1,\ldots,e_p$ can be used to show that the $e_p$ becomes negative 
when $e_p(0)=0$ and  $\alpha=2\beta$ is small enough. This is presented in the following proposition.

\begin{proposition}
\label{prop:polyn}
Let $\alpha\geq 0$ and $x_0\in \bar{\mathcal{S}}_p^+$. 
\begin{itemize}
\item[(i)]  Suppose $0<\alpha<p-1$, $\alpha\not\in B$ and $rk(x_0)<p$. Then $(x_0,\frac{\alpha}{2})\not\in W$.
In particular, the classical Wallach set $W_0=\frac12 B \cup [\frac{p-1}2,\infty)$.
\item[(ii)] If  $\alpha\in B, rk(x_0)<p$ and   $(x_0,\frac{\alpha}{2})\in W$, then $rk(x_0)\le\alpha$.
\end{itemize}
\end{proposition}

\begin{proof}
To deal with (i) suppose that $(x_0,\frac{\alpha}{2})\in W$, so there exists a solution  $X_t$ of \eqref{eq:Wishart:SDe} such that $X_t\in \bar {\mathcal S}^+_p$ for every  $t\ge 0$. The condition $rk(x_0)<p$ is equivalent to $\lambda_1(0)=0$ as well as $e_p(0)=0$. Formula \eqref{eq:polynom_last:SDEs} shows that $e_p(X_t)$ is a BESQ$^{\alpha-p+1}(0)$ in $\R$
(the superscript of a BESQ denotes its dimension), starting from $0$ with a time change by $A_t=\int_0^t e_{p-1}(s) ds$. As it was shown in \cite{bib:gjy}, the squared Bessel process with negative dimension starting from $0$ is just -BESQ$^{|\alpha-p+1|}(0)$ and consequently it becomes strictly negative just after the start. Thus, if the time change $A_t>0$ (with positive probability) then $\pr(e_p(X_t)<0)=1$ for every $t>0$. Thus $A_t\equiv 0$ and consequently the process $e_{p-1}(t)$ is always zero. Looking at the SDE \eqref{eq:polynom_first:SDEs} for $e_{p-1}$ we deduce from $e_{p-1}(t)\equiv 0$ that the drift term must vanish, which means that $e_{p-2}(t)\equiv 0$. Note that we use here the fact that $\alpha\not\in B$ implies that the factors $\alpha-n+1$ in the drift term are non-zero. Consequently, by induction, we arrive at $e_1\equiv 0$, which is impossible because the drift term of the process $e_1$ is $pdt\not=0$. The classical Wallach set corresponds to $x_0=0$. This proves (i).

The proof of (ii) is the same, however the condition $\alpha\in B$ implies that the factors $\alpha-n+1$ in the drift term are non-zero until $n=\alpha+1$. By induction, we get $e_{\alpha+1}(0)=0$, which is equivalent to $rk(x_0) \le \alpha$.

\end{proof}

Observe that the method of polynomials does not apply to the case $rk(x_0)=p$. We will show below that in this case,  the  eigenvalue process $\lambda_1(t)$ of any matrix solution of \eqref{eq:Wishart:SDe} becomes negative.

%%%%%%%%%%%%%%%%%%%%%%%%%%%%%%%%%%%%%%%%%%%%%%%%%%%%%%%%%%%%%%%%%%%%%%%%%%%%%%

\subsection{Squared Bessel particle systems}

We call  {\it squared Bessel particle system} the  eigenvalue process $\lambda_1(t)\le \cdots \le \lambda_p(t)$ of a matrix solution of the SDE \eqref{eq:Wishart:SDe}. The following corollary of the results of \cite{bib:gm2} is needed in the proof of the characterization of $W$. 

\begin{proposition}\label{bad_p}
   For  $\alpha \in B$ and any $x_0 \in {\mathcal S}^+_p$ (i.e. $rk(x_0)=p$) the eigenvalue process $\lambda_1(t)\le \cdots \le \lambda_p(t)$ is a strong and pathwise unique solution of the following SDE system 
\begin{eqnarray}
\label{eq:eigenvalues:SDe}
  d \lambda_i = 2\sqrt{|\lambda_i|}dB_i + \left(\alpha+\sum_{k\neq i}\frac{|\lambda_i|+|\lambda_k|}{\lambda_i-\lambda_k}\right)dt\/,\quad i=1,\ldots,p
\end{eqnarray}
Moreover, the eigenvalues $\lambda_i(t)$ never collide if $t>0$. 
\end{proposition}

\begin{proof}

We use a natural bijection between  the polynomials $e=(e_1,\ldots,e_p)$  and the   eigenvalues  $(\lambda_1\ldots  \lambda_p)$  belonging to the closed Weyl chamber $\bar C_+= \{(x_1,\ldots,x_p)\in\R^p: x_1\le x_2<\ldots\le x_p\}$, see  \cite{bib:gm2}.

 The first part of the proof of Prop. 4.3 in \cite{bib:gm2} together with the proof of Th. 4.4 in \cite{bib:gm2} imply that 
 any solution  of the system \eqref{eq:polynom_first:SDEs} and \eqref{eq:polynom_last:SDEs} becomes non-colliding for every $t>0$, i.e. $\lambda_i(e(t))$ are all different. Then we apply Remark 5.2\cite{bib:gm2}, which allows to construct a non-colliding  solution to the system \eqref{eq:eigenvalues:SDe}. By Th. 5.3 in \cite{bib:gm2} we get pathwise uniqueness for the solutions of \eqref{eq:eigenvalues:SDe}. From the other side, it was shown in Theorem 3 in \cite{bib:gm11} that whenever the eigenvalues are initially different then \eqref{eq:eigenvalues:SDe} holds. To deal with starting from collision points we first write SDEs for symmetric polynomials (what can be done in every case).  As it was done in \cite{bib:gm2}, using It\^o formula argument (the form of the SDEs does not depend on the starting point), we claim that the eigenvalues are solutions to \eqref{eq:eigenvalues:SDe}.

\end{proof}

\begin{proposition}
\label{prop:hit:zero}
Suppose $\alpha<p-1$ and let $(\lambda_1,\ldots,\lambda_p)$ be a non-colliding solution of the system  \eqref{eq:eigenvalues:SDe} with $\lambda_1(0)\ge 0$. Then $\pr(\lambda_1(t)<0)>0$, for every $t>0$.
\end{proposition}

\begin{proof}
Let $\Lambda = (\lambda_1,\ldots,\lambda_p)$ be a non-colliding solution to \eqref{eq:eigenvalues:SDe} and let $\tilde{\lambda_1}$ be a solution to the SDE given by
\begin{eqnarray*}
 d\tilde{\lambda}_1  = 2\sqrt{|\tilde{\lambda}_1|}dB_1+(\alpha-(p-1))dt\/.
\end{eqnarray*}
starting from $\tilde{\lambda}_1(0)=\lambda_1(0)\geq 0$. Note that $B_1$ is the same Brownian motion that appears in the SDE for $\lambda_1$. Now, we apply the techniques of local times proposed by Le Gall in \cite{bib:LeGall1983}, described also in \cite{bib:ry99}. More precisely, by Lemmas 3.3 and 3.4 in \cite{bib:ry99}, the local time $L^0(\tilde{\lambda}_1-\lambda_1)$ is zero. Using Tanaka's formula (see proof of Thm 3.7 in \cite{bib:ry99}) we obtain
\begin{eqnarray*}
  \ex(\lambda_1(t)-\tilde{\lambda}_1(t))^{+} = \ex\int_0^t \mathbf{1}_{\{\lambda_1(s)>\tilde{\lambda}_1(s)\}}\left(p-1+\sum_{k=2}^p\frac{|\lambda_1(s)|+|\lambda_k(s)|}{\lambda_1(s)-\lambda_k(s)}\right)ds\leq 0\/.
\end{eqnarray*}
The last inequality follows form the estimate ${|x|+|y|}\geq x-y$ valid for every $y<x$. It implies that
$
\pr(\lambda_1(t)\leq \tilde{\lambda}_1(t) \textrm{ for every $t\geq 0$})=1\/.
$
Since the process $\tilde{\lambda}_1$ is a squared Bessel motion with dimension $\alpha-(p-1)< 0$,  it crosses zero a.s. and next it remains in the negative half-line, see \cite{bib:gjy}. This ends the proof.
\end{proof}

\begin{proposition}\label{rank_p}
  Suppose that  $0<\alpha<p-1$ and   $rk(x_0)=p$. Then $(x_0,\alpha/2)\not\in W$.
\end{proposition}

\begin{proof}

Consider a solution  of \eqref{eq:Wishart:SDe}. The corresponding eigenvalue process $\Lambda=(\lambda_1,\ldots,\lambda_p)$ is  a strong and pathwise unique solution of \eqref{eq:eigenvalues:SDe}. If  $\alpha\in \R^+\setminus B$, it is justified  by  Cor.6.6\cite{bib:gm2}
(note a  misprint in the formulation of  Cor.6.6\cite{bib:gm2}: it should be $\alpha\in \R^+\setminus B$). In the case $\alpha\in B$, $rk(x_0)=p$, it follows from Proposition  \ref{bad_p}. Finally, by Proposition \ref{prop:hit:zero},  the first eigenvalue $\lambda_1(t)$ becomes strictly negative, so $(x_0,\alpha/2)\not\in W$. 

\end{proof}

By Propositions \ref{prop:polyn} and \ref{rank_p} we claim that

\begin{theorem}
The Mayerhofer Conjecture is true.
\end{theorem}

As a consequence, we obtain the following

\begin{corollary}
 Let $x_0\in\bar {\mathcal  S}_p^+$.  The SDE \eqref{eq:Wishart:SDe} has a solution $X_t\in   \bar {\mathcal  S}_p^+$ if and only if 
 $$
 (\alpha\in  [{p-1},\infty), x_0\in \bar  {\mathcal   S}_p^+)  \ {\rm or}\ (\alpha \in B, rk(x_0)\le \alpha).
 $$
\end{corollary}

{\bf Remark}.  In Proposition  \ref{prop:polyn}(i) we gave a stochastic proof of the characterization of the classical "central" Wallach set.  Observe that  one more stochastic proof of the central Wallach set may be given by comparison methods. Indeed,  
 it follows from   Proposition \ref{prop:hit:zero} that, if $\alpha<p-1$, $\alpha\not\in B$ and $x_0=0$, then the first eigenvalue $\lambda_1(t)$ becomes strictly negative, so $\alpha/2\not\in W_0$.

% etc, etc

% \appendix

% \section{}
% \label{}

% The Acknowledgements are an un-numbered section
%\section*{Acknowledgements}
 %Acknowledgements text here

\end{document}